\documentclass{article}
\usepackage{graphicx} 
\usepackage{amsmath,amssymb,amsthm}
\usepackage{mathtools}
\usepackage{authblk}
\usepackage[hidelinks]{hyperref}
\title{An Aggregated SIR Model \\
for Spatial Epidemic Propagation}
\date{}

\author[1]{M. Soledad Aronna}
\author[2,3]{Mariana Bergonzi}

\author[2,3]{Ernesto Kofman}

\affil[1]{School of Applied Mathematics, FGV EMAp, Brazil}

\affil[2]{French-Argentine International Center for Information and System Sciences (CIFASIS),CONICET, Argentina}
\affil[3]{FCEIA, Universidad Nacional de Rosario, Argentina}

\newtheorem{theorem}{Theorem}

\usepackage[usenames, dvipsnames]{color}

\usepackage{appendix}

\usepackage{pgf,tikz}
\usepackage{natbib}

\newcommand{\sat}{\mathrm{sat}}

\newcommand{\pS}{p_{\mathrm S}}
\newcommand{\pI}{p_{\mathrm I}}
\newcommand{\pR}{p_{\mathrm R}}
\newcommand{\pRz}{p_{\mathrm R 0}}
\newcommand{\dpS}{\dot p_{\mathrm S}}

\newcommand{\dpR}{\dot p_{\mathrm R}}
\newcommand{\Sinf}{S_{\infty}}
\newcommand{\Rinf}{R_{\infty}}
\newcommand{\Rz}{\mathcal{R}_0}

\newtheorem{thm}{Theorem}[section]

\newtheorem{prop}[thm]{Proposition}

\begin{document}
\maketitle


\abstract{We propose an extension of the classical susceptible–infectious–recovered (SIR) model that incorporates the effects of spatial propagation of an epidemic through a small number of additional compartments. The model is designed to capture the dynamics of disease spread across multiple interconnected cities or populated regions, while avoiding the high dimensionality and large parameter sets typical of network-based or agent-based approaches. Instead of explicitly modeling individual locations or mobility networks, we introduce aggregate variables that describe whether the epidemic has not yet reached, is currently active in, or has already passed through different regions of the spatial domain. This formulation allows the model to reproduce key qualitative features observed in aggregated incidence data, such as prolonged plateaus and multiple infection waves arising from asynchronous local outbreaks. The resulting system consists of ordinary differential equations with a relatively small number of interpretable parameters, providing a tractable framework for analytical investigation and numerical simulation. Our approach offers a parsimonious alternative for studying spatially structured epidemic dynamics when only aggregated data are available or when model simplicity is essential.}

\paragraph{Keywords}
SIR Model, Spatial Disease Spread, Aggregated SIR Model, Basic Reproduction Number, Epidemics on Networks

\maketitle

\section{Introduction}

The aim of this work is to extend the classical SIR (susceptible-infectious-recovered) model by including new compartments that represent the level of {\em spatial propagation} of the epidemic.
The main objective is to find a system that can imitate the dynamics observed when a disease spreads in different interconnected cities or populated regions but, differently to what is done in network-based models (see  references below), we aim at finding a set of equations involving a relatively small quantity of parameters.



Models for spatial propagation of diseases have existed for a quite long time already. Early work by  \cite{may1984spatial} from the 80s  considered the  impact of spatial heterogeneity on the vaccination requirement for the eradication of a canonical infectious disease. The latter was based on previous studies, by \cite{hethcote1978immunization} and  \cite{post1983endemic}, that noted the importance of hetereogenity in epidemiological modeling. 
More references and historical context of foundational articles can be found in  \cite{rhodes1996persistence} and \cite{bolker1995space}.

 \cite{sattenspiel1995structured} and  \cite{arino2003multi} explored how regional mobility and migration can be incorporated into epidemic models. In particular, \cite{arino2003multi} emphasized the use of space-discrete models and derived inequalities for the basic reproduction number, \(\mathcal{R}_0\).
Considering the structure of urban connectivity networks,  \cite{colizza2007reaction} and  \cite{colizza2008epidemic} developed a framework to calculate global thresholds for epidemic invasion. Complementary to these works,  \cite{pastor2001epidemic} and  \cite{pastor2002epidemic} made seminal contributions to the study of epidemic dynamics on complex networks of individuals. Recently,  \cite{aronna2024optimal} conducted a study addressing how intercity commuting patterns within metropolitan regions influence the spread of infectious diseases.



Another widely used method to take into account spatial spread is modeling through partial differential equations (PDE) with a diffusion term. Some of the initial works involving this type of model were employed for rabies \citep{murray1986spatial,kallen1985simple}, and were derived as simple extensions of the ordinary differential equations (ODE) models by adding a diffusion operator representing the random movement of the infectious species. An extensive review of PDE modeling for infectious diseases can be found in {\em e.g.} \cite{hoang2025differential}.


Agent-based simulation models provide an alternative to equation-based systems and, in some cases, can represent richer scenarios than the latter \citep{hunter2018comparison}. In an agent-based framework, each individual agent can be assigned distinct attributes and decision rules, enabling the model to capture interactions and behaviours at the individual level \citep{marshall2015formalizing,hunter2017taxonomy}.
In contrast, the approach taken in this article differs substantially from agent-based modelling: the defining feature of our model is aggregation rather than individual-level representation.


In this article, we propose an ODE model that extends the classical SIR framework by adding variables that represent spatial propagation.
Loosely speaking, the total population is assumed to live in several interconnected districts and is partitioned according to whether the epidemic is active in each location, has not yet arrived, or has already passed. Rigorous definitions are provided in the next section.
The idea is to capture the essential features of spatial diffusion in a network of cities or regions, where the epidemic begins in one location and subsequently spreads to others, generating multiple infectious waves. When examining aggregated incidence data, one often observes a curve that does not exhibit a single peak but instead displays a kind of {\em plateau,} sustained as long as different nodes reach their peaks at different times.

The article is organized as follows. In Section \ref{SecModel} , we introduce the model, explain the meaning and interpretation of the variables and parameters, and discuss the value of the basic reproduction number $\mathcal{R}_0$ for the proposed framework. Section \ref{SecExamples} presents numerical simulations that validate our findings. Detailed derivations of the basic reproduction number for both the novel aggregated model and the network model are provided in the Appendix.

\section{Model Formulation}
\label{SecModel}

In this section, we introduce the proposed epidemic model and describe its fundamental dynamical properties. We first present the system of ordinary differential equations governing the evolution of the state variables and discuss the role and epidemiological interpretation of the model parameters. Particular attention is given to the mechanisms through which spatial propagation is incorporated into the framework. We then comment on the derivation of the basic reproduction number, $\mathcal{R}_0$, which serves as a key threshold parameter characterizing the potential for epidemic outbreak and persistence, and which corresponds to the one for a classical $SIR$ model, as established in Theorem \ref{ThmR0} below.

\subsection{Model Dynamics}
We consider an epidemic taking place in a population of size $N$ that is divided in certain number of districts. The individuals remain most of the time in their own district, but they sometimes move towards other districts where they can get infected or, in case they are already infected, they can infect other individuals. 

We shall first divide the population into three main groups $\pS,$ $\pI$ and $\pR$, according to the following description:
\begin{itemize}
    \item $\pS(t)$ is the population living in districts that were not yet reached by the epidemic at time $t$;
    \item $\pI(t)$ is the population living in districts where the epidemic is active at time $t$;
   \item $\pR(t)$ is the population living in districts where the epidemic has already finished.
\end{itemize}
That way,
\begin{equation} \label{eq:N}
    \pS(t)+\pI(t)+\pR(t)=N
\end{equation}
Notice that 
\begin{itemize}
    \item the population $\pS(t)$ is composed of susceptible individuals;
    \item the population $\pR(t)$ is composed of recovered individuals, but also of susceptible individuals (not every individual is reached by the epidemic); 
    \item the population $\pI(t)$ is composed of susceptible, infected, and recovered individuals, denoted by $S(t)$, $I(t)$, and $R(t)$, respectively. Thus, 
    \begin{equation} \label{eq:pI}
        S(t)+I(t)+R(t)=\pI(t);
    \end{equation}
   \item the population $\pI(t)$ receives susceptible individuals that leave $\pS(t)$ at a certain rate $\dpS(t)$ as the corresponding districts are reached by the epidemic;
   \item the population  $\pR(t)$ receives susceptible and recovered individuals that leave $\pI(t)$ at a certain rate $\dpR(t)$ as the corresponding districts recover, this is, as the wave leaves those districts; 
   \item we let $\Sinf$ denote the fraction of individuals that remain susceptible after the epidemic finishes in a district and $\Rinf \coloneq 1-\Sinf$  the fraction of individuals that is reached by the epidemic. Then, the rate at which susceptible and recovered individuals move towards $\pR(t)$ can be computed as $\Sinf \dpR(t)$  and $\Rinf \dpR(t)$, respectively.
\end{itemize}

We shall assume that, within the population $\pI(t)$, the dynamics is governed by a modified SIR model, which takes into account that the arrival of susceptible population from $\pS$ occurs at rate $\dot \pS$ and the departure of susceptible and recovered individuals toward $\pR$. Thus, the model for the population $\pI(t)$ can be written as 

\begin{equation} \label{eq:SIRp}
    \begin{split}
        \dot S(t)&=-\dpS(t)-\beta \frac{S(t) I(t)}{\pI(t)}-\Sinf \dpR(t)\\
        \dot I(t)&=\beta \frac{S(t) I(t)}{\pI(t)}-\gamma I(t)\\
        \dot R(t)&=\gamma I(t)-\Rinf \dpR(t) \\
    \end{split}
\end{equation}

Figure~\ref{fig:psir} shows the scheme for the different compartments. 

\begin{figure}[htb]
    \centering
    \includegraphics[width=5cm]{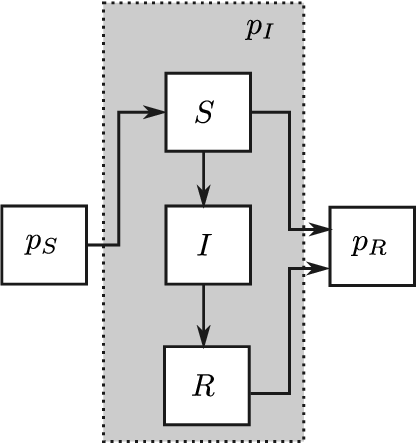}
    \caption{Aggregated model scheme.}
    \label{fig:psir}
\end{figure}

In order to complete the model, we need to establish the expressions for $\dpS(t)$ and $\dpR(t)$ that represent the advance of the epidemics through the different districts and their subsequent recovery. Regarding these terms, we make the following observations and assumptions.
\begin{itemize}
    \item The rate $\dpS(t)$ at which the epidemic reaches new districts is proportional to the number $I(t)$ of infected individuals.
    \item 
 That rate also depends on the population $\pS(t)$ living in districts that have not yet been
reached by the epidemic . However, when $\pS(t)$ is sufficiently
large, the rate of spatial advance may fail to grow proportionally to $\pS(t)$, since
it is ultimately driven by the fraction of this population that is in contact with
districts where the epidemic is already active. Consequently, a reasonable expression
for the dynamics of $\pS(t)$ is
\begin{equation} \label{eq:dpS}
    \dpS(t) = - \rho\, I(t)\, \sat\!\left( \alpha \frac{\pS(t)}{N} \right),
\end{equation}
where $\sat(\cdot)$ is a \emph{saturation function} taking values in the interval $[0,1]$.

 We expect that a condition of the following type holds:
 \[
     \sat(x) \approx
     \begin{cases}
         1&\text{if }x>1,\\
         x&\text{otherwise}.
     \end{cases}
 \]
A smooth approximation of that function is given by 
 \begin{equation} \label{eq:sat2}
     \sat(x)=\frac{2}{\pi}\arctan\left(\frac{\pi}{2} x\right)
 \end{equation}
which is more adequate for its use in numerical algorithms.

    \item Susceptible individuals in $p_I$ are driven either to $I$ or to $p_R$, where herd immunity has been achieved. Thus, we can think that recovered individuals in $p_I$ move susceptible individuals towards $p_R$ so that
\begin{equation}\label{eq:dS}
     \dot S(t)=-\dpS(t)-\beta \frac{S(t) I(t)}{\pI(t)}-\beta_R \frac{S(t) R(t)}{\pI(t)}
 \end{equation}    
   for certain parameter $\beta_R$.

   \item Comparing this expression with that of Eq.\eqref{eq:SIRp}, results in
\begin{equation} \label{eq:dpR2}
     \dpR(t)=\frac{\beta_R}{\Sinf} \frac{S(t) R(t)}{\pI(t)}
 \end{equation}
\end{itemize}
Then, combining Eqs.\eqref{eq:pI} and \eqref{eq:dpR2} with Eq.\eqref{eq:SIRp}, the complete model equations result as follows
\begin{equation} \label{eq:aggmodel}
    \begin{split}
        \dpS(t)&=-\rho I(t) \sat(\alpha \frac{ \pS(t)}{N})\\
        \dot S(t)&=\rho I(t) \sat(\alpha \frac{ \pS(t)}{N})-\beta  \frac{S(t) I(t)}{\pI(t)}-\beta_R  \frac{S(t) R(t)}{\pI(t)}\\
        \dot I(t)&=\beta \frac{S(t) I(t)}{\pI(t)}-\gamma I(t)\\
        \dot R(t)&=\gamma I(t)-\frac{\beta_R \Rinf}{\Sinf} \frac{S(t) R(t)}{\pI(t)}\\
        \dpR(t)&= \frac{\beta_R}{\Sinf} \frac{S(t) R(t)}{\pI(t)}
    \end{split}
\end{equation}
with $\pI(t)=S(t)+I(t)+R(t)$.

\subsection{About the basic reproduction number $\mathcal{R}_0$}

Before discussing the model parameters, we compute the \emph{basic reproduction
number} $\mathcal{R}_0$ associated with our model, following the classical approach
of \cite{VanDenDriessche2002}. Under a set of hypotheses, which are rigorously verified
in Appendix~\ref{AppR0}, the procedure developed in \cite{VanDenDriessche2002} applies
to our setting. The full derivation of $\mathcal{R}_0$ is provided in
Appendix~\ref{AppR0}. The result is stated below.

\begin{theorem}
\label{ThmR0}
The basic reproduction number for model~\eqref{eq:aggmodel} is given by
\begin{equation}
    \mathcal{R}_0 = \frac{\beta}{\gamma},
\end{equation}
which coincides with the basic reproduction number of the standard SIR model
(see, for example, \cite{brauer2012mathematical}).
\end{theorem}

\subsection{About the model parameters}

The model given by equation~\eqref{eq:aggmodel} involves seven parameters, namely
$\beta, \gamma, \rho, \alpha, \beta_R, \Rinf, \Sinf$. As usual, $\beta$ denotes
the transmission rate and $\gamma$ the recovery rate.

The parameters $\Rinf$ and $\Sinf = 1 - \Rinf$ represent the fractions of individuals
who are recovered or remain susceptible, respectively, after the epidemic has
completely passed through a district. In the standard SIR model, these quantities
are defined as the limits, as time tends to infinity, of the recovered and susceptible
populations.
Accordingly, $\Rinf$ can be computed as in the classical SIR framework as the solution of the transcendental equation
\begin{equation}
    \Rinf = 1 - e^{-\mathcal{R}_0 \Rinf},
    \qquad \text{with } \mathcal{R}_0 = \frac{\beta}{\gamma}.
\end{equation}

In conclusion, there are only 3 additional parameters to those of a standard SIR model: 

\begin{itemize}
    
    \item $\rho$ and $\alpha$ are related to the mobility of the population and the degree of connectivity between districts;

    \item $\beta_R$ defines the rate at which the population moves towards the recovered compartment $\pR$. Thus, it is related to the duration of the outbreak in the districts, which depends in turn on $\beta$, $\gamma$ and the population size of each district. 
    

\end{itemize}

\section{Numerical Examples}\label{SecExamples}
In this section we present two numerical examples. In the first one, the aggregated model is used to reproduce the simulation results of a synthetic metapopulation model. In the second one, the model is utilized to fit the real data corresponding to the daily cases of Covid-19 reported in Argentina. 

The reported simulation results can be reproduced using the Python routines available here: \url{https://fceia.unr.edu.ar/~kofman/files/pSIR_Agg.zip}.

\subsection{Synthetic metapopulation model}
We consider a population divided into $D$ districts with equal population size $N_i = N / D$, for each $i=1,\dots, N$. The dynamics in each region is governed by the following set of equations 
\begin{equation} \label{eq:metapop}
    \begin{split}
        \dot S_i(t)&=-\beta \frac{S_i(t)\hat I_i(t)}{N_i}\\
        \dot I_i(t)&=\beta \frac{S_i(t)\hat I_i(t)}{N_i}-\gamma I_i(t)\\
        \dot R_i(t)&=\gamma I_i(t)
    \end{split}
\end{equation}
for $i=1,\ldots,D$, with
\begin{equation}\label{hatI}
    \hat I_i(t)=
    \begin{cases}
        (1-p)I_1(t)+p I_2(t) & \text{ if }i=1\\
        (1-p)I_D(t)+p I_{D-1}(t) & \text{ if }i=D\\
        (1-2p)I_i(t)+p I_{i-1}(t)+p I_{i+1}(t) & \text{ otherwise}
    \end{cases}
\end{equation}
From latter equation, we see that in the network model \eqref{eq:metapop}, each district has two neighboring districts, except at the boundaries (i.e., for \(i=1\) and \(i=D\)).

This model, which can be regarded as a simplification of the one introduced in \cite{aronna2024optimal}, represents the situation in which individuals spend a fraction of time $p$ in each neighboring district and the remainder of the time in their own district. The model is also a particular case of the one presented in Appendix~\ref{AppR0network}, where it is proven that the basic reproduction number for the network model is $\hat{\mathcal{R}}_0=\beta/\gamma$, independently of the mobility matrix.

 For the numerical simulations, we set the parameters to \(D=10\), \(p=0.005\), \(\beta=1\), and \(\gamma=0.5\). The population consists of \(10^6\) individuals, but we normalize to have \(N=1\) for a better understanding of the proportions. The initial condition corresponds to an epidemic originating from a single infected individual in the first district, that is, \(I_1(t=0)=10^{-6}\).

Figure~\ref{fig:metapop} shows the simulated system trajectories for the different districts and the resulting sum of infected individuals. Notice that this sum exhibits a wide plateau rather than a peak, a behavior that is not captured by standard compartmental models, as already mentioned above. Actually, we chose this particular network structure in order to produce this type of evolution, where the peaks of the different districts occur at different times inducing an aggregated trajectory that behaves radically different from the trajectories of standard compartmental models.

\begin{figure}[htb]
    \centering
    \includegraphics[width=\linewidth]{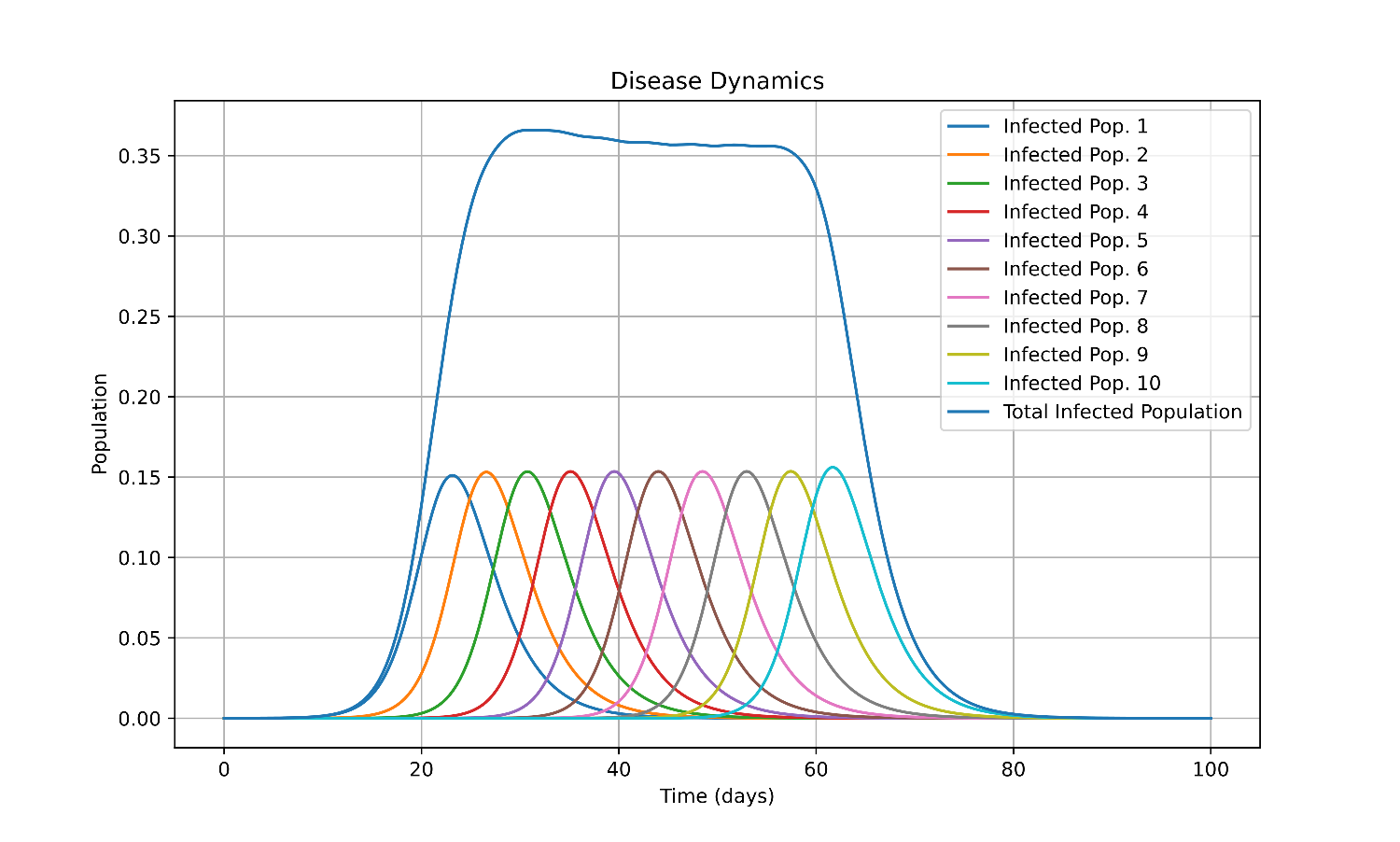}
    \caption{Infected population trajectories in the Multi-Population model.}
    \label{fig:metapop}
\end{figure}

The aggregated model of Eq.\eqref{eq:aggmodel} was then simulated with parameters $N=1$, 
$\beta = 1$, $\gamma = 0.5$,
$\rho =  0.657$, 
$\beta_R =  0.152$, and
$\alpha =  32.81$
using the saturation function of Eq.\eqref{eq:sat2}. 
The initial conditions were chosen as follows. We set $I(0)=10^{-6}$ to represent a single initial infection, and $R(0)=\pR(0)=0$. The value $\pI(0)=0.105$ was used and then we set
$S(0)=\pI(0)-I(0)$ and $\pS(0)=N-\pI(0)$.
The reported values for the initial state $\pI(0)$ and the model parameters $\rho$, $\beta_R$ and $\alpha$ were obtained by a simple optimization procedure (using Python \verb|scipy.optimize.curve_fit|) to fit the data of the total amount of infected individuals.

\begin{figure}[htb]
    \centering
    \includegraphics[width=\linewidth]{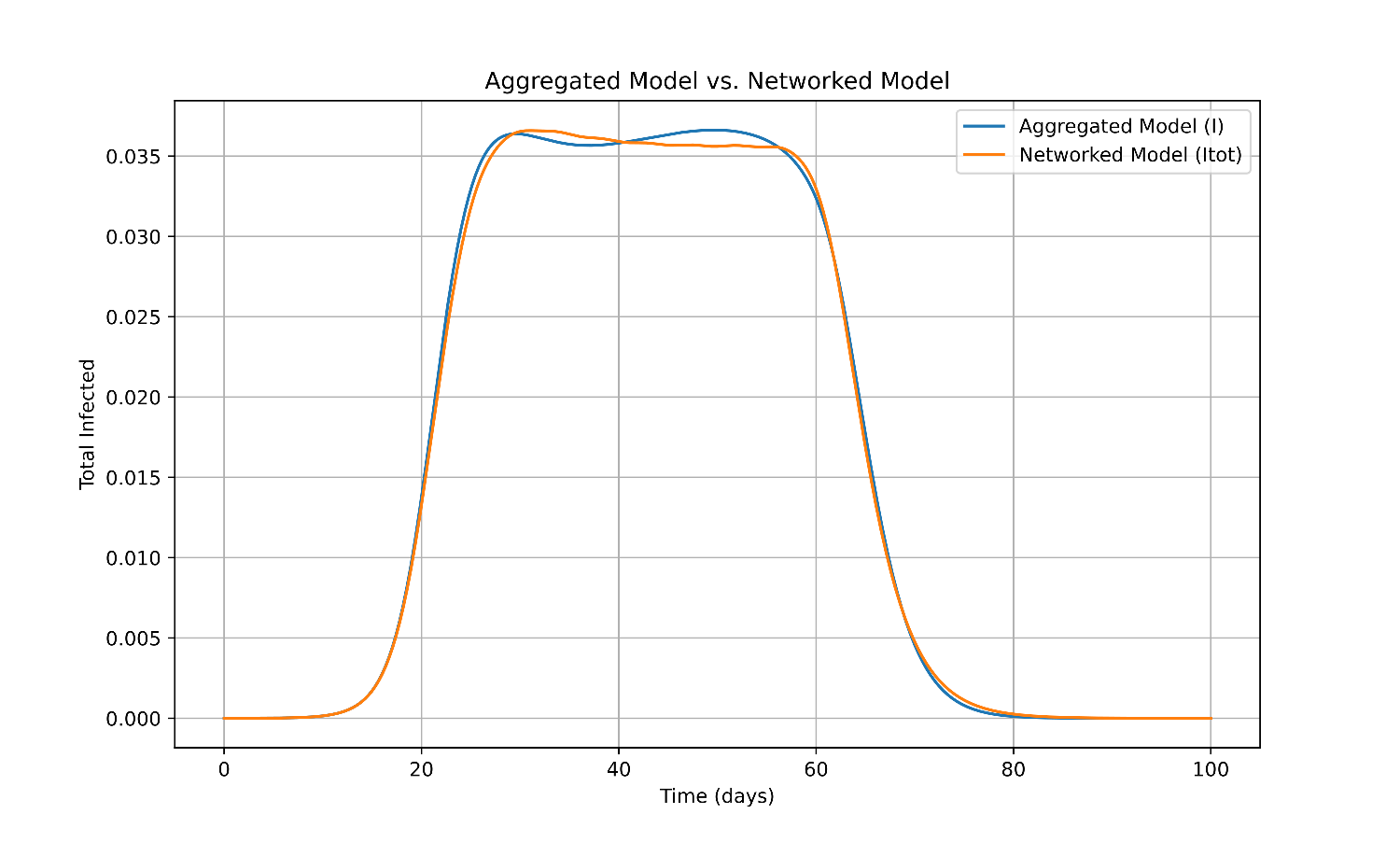}
    \caption{Infected population trajectories in the Aggregated and Multi-Population model.}
    \label{fig:aggregated}
\end{figure}

It can be observed that the aggregated model of Eq.\eqref{eq:aggmodel} can accurately reproduce the result of the metapopulation model of Eq.\eqref{eq:metapop}. Moreover, the initial value $\pI(0)=0.105\approx 1/D$   is consistent with the fact that the outbreak started at a single district that contains the $10\%$ of the total population. 

\subsection{Fitting real data}
The second example corresponds to the first wave of COVID-19 in Argentina that occurred in the period March--November 2020. The orange curve in Figure~\ref{fig:fitting} shows the daily number of reported cases (processed  with a moving average filter of length $W=7$). The raw data can be downloaded from \url{https://covidstats.com.ar}, and the number of daily cases is available at \url{https://fceia.unr.edu.ar/~kofman/files/nacion.csv}.

In this case, the curve exhibits an approximately linear growth for a long period of time (consistent with the fact that the epidemics was reaching new districts) after which it experiences an abrupt fall. Like in the previous example, this type of behavior is not captured by standard compartmental models.

In order to normalize the data for a population of $N=1$, the number of daily cases was divided by the total population of Argentina at that date $P=45\cdot 10^6$ and multiplied by $d=10$ to take into account the under-detection factor.  This factor can be roughly inferred by comparing the case fatality rate observed in Argentina 

Figure~\ref{fig:fitting} also shows the simulation results obtained using the aggregated model with parameters 
$\rho =  0.8026$, $\beta_R =  0.01911$, $\alpha =  242.4$ $, \beta =  0.2118$ and $\gamma =  0.1667$. The initial values were set to $I(0)=1.43\cdot 10^{-4}$ (in order to coincide with the initial data), $\pI(0)=0.1796$, $\pR(0)=R(0)=0$, $\pS(0)=N-\pI(0)$, $R(0)=\pR(0)=0$. 
The reported values of $\rho$, $\beta_R$, $\alpha$, $\beta$, and $p_I(0)$  were obtained by a simple optimization procedure (using Python \verb|scipy.optimize.curve_fit|), as in the first example. 

\begin{figure}[htb]
    \centering
    \includegraphics[width=\linewidth]{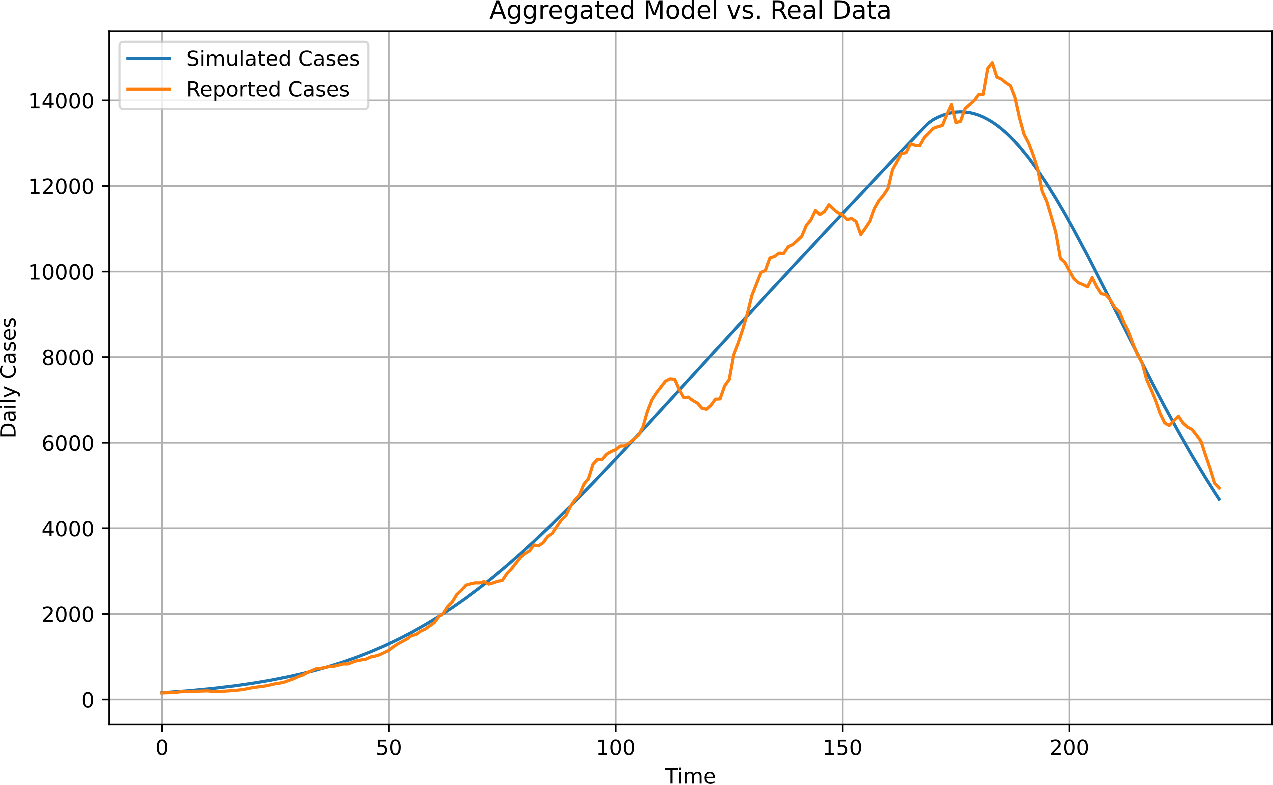}
    \caption{Daily cases of COVID-19 in Argentina (with 7-day moving average) and simulation results with the aggregated model.}
    \label{fig:fitting}
\end{figure}

As before, the aggregated model was able to accurately fit the data and the optimization procedure led to reasonable values for all the parameters.

\section{Conclusions}
We introduced a new SIR model with two additional compartments, $\pS$ and $\pR.$ The first one contains the population living in districts where the epidemics has not yet arrived and the latter one contains the population where the epidemics has already finished. That way, this new model attempts to capture the aggregated behavior of metapopulation dynamics in a simple and direct manner by adding only two state variables to those of the standard SIR model.   

The aggregated model contains only 7 parameters, where 4 of them correspond to those of the classic SIR model ($\beta$, $\gamma$, $\Rinf$, and $\Sinf$) while the remaining 3 are related to the network connectivity ($\alpha$), the population mobility ($\rho$) and the duration of the outbreak in a district ($\beta_R$). We showed that, like in the standard SIR model, the basic reproduction number is given by $\Rz=\beta/\gamma$. 

Two numerical examples were analyzed. The first case corresponds to a synthetic metapopulation model that exhibits a plateau in the sum of the infected individuals of the different districts. The second example deals with the reported cases of Covid-19 in Argentina, where it is observed that the linear growth of the daily cases follows an abrupt fall. Both types of dynamics can be explained by the fact that the epidemics reaches different districts at different times. While standard compartmental models are unable to show these types of evolutions, the proposed model could successfully reproduce both situations.

The work opens several future lines of research. First, it would be desirable, if possible, to establish a quantitative relation between the parameters $\alpha$ and $\rho$ of the aggregated model with those of the network structure and mobility. That relationship would not only simplify the model parametrization but it may also allow to infer some network features based on the aggregated behavior. Similarly, a quantitative relation between the value of $\beta_r$ in the aggregated model with those of $\gamma$, $\beta$ and the size of the different districts would further simplify the parametrization.

The idea used here for the SIR model could  be extended to other compartmental models like SEIR and SIR-SI (as for vector-borne diseases), for instance. While in an SEIR case the extension seems to be straightforward, in the SIR-SI case the problem appears to be more involved since both populations must evolve in a spatially coordinated fashion.

A potential use of the proposed model is that of studying the effect of control actions that affect mobility between districts. Here, the use of a low order model like the one proposed in this work would simplify the design of controllers and optimization procedures with respect to the use of metapopulation models. In any case, the use of the aggregated model for control purposes would require to establish first the aforementioned relationship between the parameters of this model with those of the network dynamics.



\section*{Appendix}

\begin{appendices}

\section{Computing $\mathcal{R}_0$ for the aggregated model}\label{AppR0}

We follow the calculations proposed in \cite{VanDenDriessche2002}, to compute the {\it basic reproduction number,} denoted by $\mathcal{R}_0$ for model \eqref{eq:full}. With that aim, we consider the full model, in which we re-order the variables and put first the infectious compartment $I:$
\begin{equation} \label{eq:full}
    \begin{split}
        \dot I(t)&=\beta \frac{S(t) I(t)}{S(t)+I(t)+R(t)}-\gamma I(t)\\
        \dot S(t)&=\rho I(t) \sat(\alpha \frac{ \pS(t)}{N})-\beta \frac{S(t) I(t)}{S(t)+I(t)+R(t)}-\beta_R \frac{S(t) R(t)}{S(t)+I(t)+R(t)}\\
        \dot R(t)&=\gamma I(t)-\beta_R \frac{1-\Sinf}{\Sinf} \frac{S(t) R(t)}{S(t)+I(t)+R(t)}\\
        \dpS(t)&=-\rho I(t) \sat(\alpha \frac{ \pS(t)}{N})\\
        \dpR(t)&=\frac{\beta_R}{S_\infty} \frac{S(t) R(t)}{\pI(t)}
    \end{split}
\end{equation} 
For system \eqref{eq:full}, we use $x$ to denote the state $(I,S,R,\pS,\pR)$, and then the disease-free equilibria (DFE) take the form
\begin{equation}
    x_0 = (0, S_0, 0, 0, \pRz),
\end{equation}
with $S_0 + \pRz = N.$

For system \eqref{eq:full}, we use $x$ to denote the state $(I,S,R,\pS,\pR)$, and then the disease-free equilibria (DFE) take the form
\begin{equation}
    x_0 = (0, S_0, 0, 0, \pRz),
\end{equation}
with $S_0 + \pRz = N.$
Following the notation employed in \cite{VanDenDriessche2002},
let $x = (x_1,\dots,x_n)^T,$ with each $x_i \geq 0$, be the number of individuals in each compartment, for $i=1,\dots,5$. For clarity, we sort the
compartments so that the first $m$ compartments correspond to infected individuals. In our case, $m=1.$ And we define ${\bf X}_s$ to be the set of all disease free states:
\begin{equation}
   {\bf X}_s := \{ x\geq 0:\,x_i=0,\,\, i=1,\dots,m\}. 
\end{equation}
We write the dynamics of each compartment in the form
\[
\dot{x}_i = f_i(x) = \mathcal{F}_i(x) - \mathcal{V}_i(x),
\]
where $\mathcal{F}_i(x)$ is the rate of appearance of new infections in compartment $i,$  and each $\mathcal{V}_i$ can be written in the form $\mathcal{V}_i = \mathcal{V}_i^- - \mathcal{V}_i^+,$ with $\mathcal{V}_i^-$ being the rate of transfer of individuals {\it out} of compartment $i$ and $\mathcal{V}_i^+$ the rate of transfer of individuals {\it into} of compartment $i.$
Writing $f,\mathcal{F},\mathcal{V}^-,\mathcal{V}^+$ for the vector functions with components $f_i,\mathcal{F}_i,\mathcal{V}_i^-,\mathcal{V}_i^+$ for $i=1,\dots,5,$ respectively, we have
\begin{equation}
    \mathcal{F}(x) = 
    \begin{bmatrix}
    \beta \frac{S I}{S+I+R} \\ 0 \\ 0 \\ 0 \\ 0
    \end{bmatrix},\,
    {\mathcal{V}^-}(x) = 
    \begin{bmatrix}
    \gamma I \\ \beta \frac{S I}{S+I+R}+\beta_R \frac{S R}{S+I+R} \\ \beta_R \frac{1-\Sinf}{\Sinf} \frac{S R}{S+I+R} \\ \rho I \sat(\alpha \frac{ \pS}{N}) \\ 0
    \end{bmatrix},\,
    {\mathcal{V}^+}(x) = 
    \begin{bmatrix}
    0 \\ \rho I \sat(\alpha \frac{ \pS}{N}) \\ \gamma I \\ 0 \\ \frac{\beta_R}{S_\infty} \frac{S R}{S+I+R} 
    \end{bmatrix},
\end{equation}
$\mathcal{V} = \mathcal{V}^- - \mathcal{V}^+$ and $f = \mathcal{F} - \mathcal{V}.$
We proceed now to check that with theses functions, the system verifies the requires hypotheses, namely (following the notation in \cite{VanDenDriessche2002}):
\begin{itemize}
    \item[(A1)] If $x \geq 0,$ then $\mathcal{F}_i, \mathcal{V}_i^+,\mathcal{V}_ i^- \geq 0,$ for $i=1,\dots,5.$ 
    \item[(A2)] For each $i=1,\dots, n:$ if $x_i=0,$ then $\mathcal{V}_ i^-=0.$
    In particular, if $x \in {\bf X}_s,$ then $\mathcal{V}_ i^-=0,$ for $i=1.$
    \item[(A3)] $\mathcal{F}_i=0,$ if $i>m.$
    \item[(A4)] If $x \in {\bf X}_s,$ then $\mathcal{F}_i(x)=0$ and $\mathcal{V}_i^+(x)=0,$ for $i=1.$
    \item[(A5)] If $\mathcal{F}(x)$ is set to zero, then all eigenvalues of $Df(x_0)$ have negative real parts.
\end{itemize}

Hypotheses (A1)-(A4) can be verified immediately. Let us check (A5). We have
\begin{equation}
    \begin{split}
    D\mathcal{F}(x_0) &=
    \left.
    \begin{bmatrix}
        \beta \frac{S(S+R)}{(S+I+R)^2} & \beta \frac{I(I+R)}{(S+I+R)^2} & -\beta \frac{SI}{(S+I+R)^2} & 0 & 0 \\
        0 & 0 & 0 & 0 & 0 \\
        0 & 0 & 0 & 0 & 0 \\
        0 & 0 & 0 & 0 & 0 \\
        0 & 0 & 0 & 0 & 0 
    \end{bmatrix}\right|_{x=x_0}\\
    &= 
    \begin{bmatrix}
        \beta  & 0 & 0 & 0 & 0 \\
        0 & 0 & 0 & 0 & 0 \\
        0 & 0 & 0 & 0 & 0 \\
        0 & 0 & 0 & 0 & 0 \\
        0 & 0 & 0 & 0 & 0 
    \end{bmatrix}
    \end{split},
\end{equation}
and
\if{
\begin{equation}
    \begin{split}
    & D\mathcal{V}(x_0) \\
    &= 
    \begin{bmatrix}
        \gamma & 0 & 0 & 0 & 0 \\
         \frac{\beta S(S+R) - \beta_R SR}{(S+I+R)^2} - \rho \,\sat(\alpha \frac{ \pS}{N})   & \beta \frac{(\beta I+ \beta_R R)(I+R)}{(S+I+R)^2}  & 
         \frac{-\beta SI + \beta_R S(S+I)}{(S+I+R)^2} 
        & -\rho\frac{\alpha}{N}\sat'(\alpha \frac{ \pS}{N}) & 0 \\
       - \beta_R \frac{1-S_{\infty}}{S_{\infty}} \frac{SR}{(S+I+R)^2} - \gamma & \beta_R \frac{1-S_{\infty}}{S_{\infty}} \frac{R(I+R)}{(S+I+R)^2} & \beta_R \frac{1-S_{\infty}}{S_{\infty}} \frac{S(I+S)}{(S+I+R)^2} & 0 & 0 \\
        \rho\,\sat(\alpha \frac{ \pS}{N}) & 0 & 0 & \rho \frac{\alpha}{N} I \sat'(\alpha \frac{ \pS}{N}) & 0 \\
        \frac{\beta_R SR}{S_{\infty}(S+I+R)^2} & -\frac{\beta_R R (I+R)}{S_{\infty}(S+I+R)^2} & -\frac{\beta_R S (S+I)}{S_{\infty}(S+I+R)^2} & 0 & 0 
    \end{bmatrix}|_{x=x_0} \\
    &=  
    \begin{bmatrix}
        \gamma & 0 & 0 & 0 & 0 \\
         0  & 0  & \beta_R 
        & -\rho\frac{\alpha}{N}\sat'(0) & 0 \\
        - \gamma & 0 & \beta_R \frac{1-S_{\infty}}{S_{\infty}}  & 0 & 0 \\
        0 & 0 & 0 & 0 & 0 \\
        0 & 0 & -\frac{\beta_R }{S_{\infty}} & 0 & 0 
    \end{bmatrix}.
    \end{split}
\end{equation}
}\fi
\begin{equation}
\begin{split}
D\mathcal{V}(x_0)
&=
\begin{aligned}
\left[
\begin{array}{cc}
\gamma & 0 \\
\frac{\beta S(S+R) - \beta_R SR}{(S+I+R)^2}
 - \rho \sat\!\left(\alpha \frac{\pS}{N}\right)
&
\beta \frac{(\beta I+ \beta_R R)(I+R)}{(S+I+R)^2} \\
-\beta_R \frac{1-S_{\infty}}{S_{\infty}} \frac{SR}{(S+I+R)^2} - \gamma
&
\beta_R \frac{1-S_{\infty}}{S_{\infty}} \frac{R(I+R)}{(S+I+R)^2} \\
\rho \sat\!\left(\alpha \frac{\pS}{N}\right) & 0 \\
\frac{\beta_R SR}{S_{\infty}(S+I+R)^2}
&
-\frac{\beta_R R (I+R)}{S_{\infty}(S+I+R)^2}
\end{array}
\right.
\\[1.2ex]
\left.\left.
\begin{array}{ccc}
0 & 0 & 0 \\
\frac{-\beta SI + \beta_R S(S+I)}{(S+I+R)^2}
&
-\rho \frac{\alpha}{N}\sat'\!\left(\alpha \frac{\pS}{N}\right)
& 0 \\
\beta_R \frac{1-S_{\infty}}{S_{\infty}} \frac{S(I+S)}{(S+I+R)^2}
& 0 & 0 \\
0 &
\rho \frac{\alpha}{N} I \sat'\!\left(\alpha \frac{\pS}{N}\right)
& 0 \\
-\frac{\beta_R S (S+I)}{S_{\infty}(S+I+R)^2}
& 0 & 0
\end{array}
\right]\right|_{x=x_0}
\end{aligned}\\
 &=  
    \begin{bmatrix}
        \gamma & 0 & 0 & 0 & 0 \\
         0  & 0  & \beta_R 
        & -\rho\frac{\alpha}{N}\sat'(0) & 0 \\
        - \gamma & 0 & \beta_R \frac{1-S_{\infty}}{S_{\infty}}  & 0 & 0 \\
        0 & 0 & 0 & 0 & 0 \\
        0 & 0 & -\frac{\beta_R }{S_{\infty}} & 0 & 0 
    \end{bmatrix}.
\end{split}
\end{equation}
Now, setting 
\begin{equation}
    F \coloneq \left[ \frac{\partial \mathcal{F}_i}{\partial x_j}(x_0)\right],\qquad
    V \coloneq \left[ \frac{\partial \mathcal{V}_i}{\partial x_j}(x_0)\right],\quad \text{with } 1\leq i,j \leq m,
\end{equation}
in view of \cite[Lemma 1]{VanDenDriessche2002}, we get that $F = \beta$ is non-negative and $V = \gamma$ is non-singular. The {\it basic reproduction number} for system \eqref{eq:full} is then given by
\begin{equation}
    \mathcal{R}_0 \coloneq \rho(F V^{-1}),
\end{equation}
where $\rho(A)$ denotes the spectral radius of a matrix $A.$ We have
\begin{equation}
    F \coloneq \frac{\partial \mathcal{F}_1}{\partial I}(x_0) = \beta,\quad 
    V \coloneq  \frac{\partial \mathcal{V}_i}{\partial I}(x_0) = \gamma,
\end{equation}
which leads to 
\begin{equation}
    \mathcal{R}_0 = \frac{\beta}{\gamma},
\end{equation}
which coincides with the basic reproduction number of the standard SIR model (see {\em e.g.} \cite{brauer2012mathematical}).

\section{Computing $\mathcal{R}_0$ for the network model}\label{AppR0network}

We consider a generalization of model \eqref{eq:metapop}-\eqref{hatI} for an arbitrary mobility matrix $P=[p_{ij}]_{1\leq i,j \leq D},$ where $p_{ij}$ represents the daily time proportion that a resident of region $i$ spends in region $j.$ The system is given by \eqref{eq:metapop}, where $\hat I_i$ has the general structure
\begin{equation}\label{hatI2}
    \hat I_i \coloneqq \sum_{j=1}^D p_{ji} I_j.
\end{equation}

\begin{prop}
    The basic reproduction number $\hat{\mathcal{R}}_0$ for the network model \eqref{eq:metapop}-\eqref{hatI2} is $\hat{\mathcal{R}}_0 = \frac{\beta}{\gamma}.$
\end{prop}

\begin{proof}
Following the calculations for $\mathcal{R}_0$ of previous subsection, we have
\begin{equation}
    \mathcal{F}(x) = 
    \begin{bmatrix}
    \beta \frac{S_1}{N_1}  \sum_{j=1}^D p_{j1} I_j \\ \beta \frac{S_2}{N_2}  \sum_{j=1}^D p_{j2} I_j \\ \vdots  \\ 
    \beta \frac{S_D}{N_D}  \sum_{j=1}^D p_{jD} I_j \\
    0 \\
    \vdots \\
    0 \\
    0 \\
    \vdots \\
    0
    \end{bmatrix},\,
    {\mathcal{V}^-}(x) = 
    \begin{bmatrix}
    \gamma I_1 \\
    \gamma I_2 \\
    \vdots \\
    \gamma I_D \\
    \beta \frac{S_1}{N_1}  \sum_{j=1}^D p_{j1} I_j \\ \vdots  \\ 
    \beta \frac{S_D}{N_D}  \sum_{j=1}^D p_{jD} I_j \\
    0 \\
    \vdots \\
    0
    \end{bmatrix},\,
    {\mathcal{V}^+}(x) = 
    \begin{bmatrix}
    0 \\
    \vdots \\
    0 \\
    0 \\
    \vdots \\
    0 \\
    \gamma I_1 \\
    \gamma I_2 \\
    \vdots \\
    \gamma I_D \\
    \end{bmatrix},
\end{equation}
We get
\begin{equation}
    F = 
    \begin{bmatrix}
        \beta \frac{S_{1,0}}{N_1} p_{11} & \beta \frac{S_{1,0}}{N_1} p_{21} & \dots & \beta \frac{S_{1,0}}{N_1} p_{D1} \\
        \vdots & \vdots & \ddots & \vdots \\
         \beta \frac{S_{D,0}}{N_D} p_{1D} & \beta \frac{S_{D,0}}{N_D} p_{2D} & \dots & \beta \frac{S_{D,0}}{N_D} p_{DD} \\
    \end{bmatrix},
     V = 
    \begin{bmatrix}
        \gamma & \dots & 0 \\
        \vdots  & \ddots & \vdots \\
        0 & \dots & \gamma
    \end{bmatrix}.
\end{equation}
On a DFE, we necessarily have $I_{i,0} = 0$ and $R_{i,0}=0,$  thus $S_{i,0} = N_i,$ for all $i=1,\dots, D.$ So we get
\begin{equation}
    F = \beta P.
\end{equation}
Consequently,
\begin{equation}
    \hat{\mathcal{R}}_0 = \rho(F V^{-1}) = \frac{\beta}{\gamma}\rho(P) =  \frac{\beta}{\gamma},
\end{equation}
for any mobility matrix, as long as it is row-stochastic (non-negative entries and, for each row, the entries sum 1).
\end{proof}

\end{appendices}

\bibliographystyle{plain}       
\bibliography{biblio}

@book{brauer2012mathematical,
  title={Mathematical models in population biology and epidemiology},
  author={Brauer, F. and Castillo-Chavez, C.},
  volume={2},
  number={10},
  year={2012},
  publisher={Springer}
}

@article{marshall2015formalizing,
  title={Formalizing the role of agent-based modeling in causal inference and epidemiology},
  author={Marshall, B.D.L. and Galea, S.},
  journal={American journal of epidemiology},
  volume={181},
  number={2},
  pages={92--99},
  year={2015},
  publisher={Oxford University Press}
}

@article{hunter2018comparison,
  title={A comparison of agent-based models and equation based models for infectious disease epidemiology},
  author={Hunter, E. and Mac Namee, B. and Kelleher, J.D.},
  year={2018},
  publisher={Technological University Dublin}
}

@article{hoang2025differential,
  title={Differential equation models for infectious diseases: Mathematical modeling, qualitative analysis, numerical methods and applications: MT Hoang, M. Ehrhardt},
  author={Hoang, M.T. and Ehrhardt, M.},
  journal={SeMA Journal},
  pages={1--36},
  year={2025},
  publisher={Springer}
}

@article{kallen1985simple,
  title={A simple model for the spatial spread and control of rabies},
  author={K{\"a}ll{\'e}n, A. and Arcuri, P. and Murray, J.D.},
  journal={Journal of Theoretical Biology},
  volume={116},
  number={3},
  pages={377--393},
  year={1985},
  publisher={Elsevier}
}

@article{murray1986spatial,
  title={On the spatial spread of rabies among foxes},
  author={Murray, J.D. and Stanley, E.A. and Brown, D.L.},
  journal={Proceedings of the Royal society of London. Series B. Biological sciences},
  volume={229},
  number={1255},
  pages={111--150},
  year={1986},
  publisher={The Royal Society London}
}

@article{hunter2017taxonomy,
  title={A taxonomy for agent-based models in human infectious disease epidemiology},
  author={Hunter, E. and Mac Namee, B. and Kelleher, J.D.},
  journal={Journal of Artificial Societies and Social Simulation},
  volume={20},
  number={3},
  year={2017},
  publisher={Jasss}
}

@article{pastor2002epidemic,
  title={Epidemic dynamics in finite size scale-free networks},
  author={Pastor-Satorras, R. and Vespignani, A.},
  journal={Physical Review E},
  volume={65},
  number={3},
  pages={035108},
  year={2002},
  publisher={APS}
}

@article{pastor2001epidemic,
  title={Epidemic dynamics and endemic states in complex networks},
  author={Pastor-Satorras, R. and Vespignani, A.},
  journal={Physical Review E},
  volume={63},
  number={6},
  pages={066117},
  year={2001},
  publisher={APS}
}

@article{rhodes1996persistence,
  title={Persistence and dynamics in lattice models of epidemic spread},
  author={Rhodes, C.J. and Anderson, R.M.},
  journal={Journal of theoretical biology},
  volume={180},
  number={2},
  pages={125--133},
  year={1996},
  publisher={Elsevier}
}

@article{bolker1995space,
  title={Space, persistence and dynamics of measles epidemics},
  author={Bolker, B. and Grenfell, B.T.},
  journal={Philosophical Transactions of the Royal Society of London. Series B: Biological Sciences},
  volume={348},
  number={1325},
  pages={309--320},
  year={1995},
  publisher={The Royal Society London}
}

@article{may1984spatial,
  title={Spatial heterogeneity and the design of immunization programs},
  author={May, R.M. and Anderson, R.M.},
  journal={Mathematical Biosciences},
  volume={72},
  number={1},
  pages={83--111},
  year={1984},
  publisher={Elsevier}
}

@article{post1983endemic,
  title={Endemic disease in environments with spatially heterogeneous host populations},
  author={Post, W.M. and DeAngelis, D.L. and Travis, C.C.},
  journal={Mathematical Biosciences},
  volume={63},
  number={2},
  pages={289--302},
  year={1983},
  publisher={Elsevier}
}

@article{hethcote1978immunization,
  title={An immunization model for a heterogeneous population},
  author={Hethcote, H.W.},
  journal={Theoretical population biology},
  volume={14},
  number={3},
  pages={338--349},
  year={1978},
  publisher={Elsevier}
}

@article{VanDenDriessche2002,
  title={Reproduction numbers and sub-threshold endemic equilibria for compartmental models of disease transmission},
  author={Van den Driessche, P. and Watmough, J.},
  journal={Mathematical biosciences},
  volume={180},
  number={1-2},
  pages={29--48},
  year={2002},
  publisher={Elsevier}
}

@article{sattenspiel1995structured,
  title={A structured epidemic model incorporating geographic mobility among regions},
  author={L. Sattenspiel and K. Dietz},
  journal={Math. Biosci.},
  volume={128},
  number={1-2},
  pages={71--91},
  year={1995}
}

@article{arino2003multi,
  title={A multi-city epidemic model},
  author={J. Arino and P. Van den Driessche},
  journal={Math. Population Stud.},
  volume={10},
  number={3},
  pages={175--193},
  year={2003}
}

@article{colizza2007reaction,
  title={Reaction--diffusion processes and metapopulation models in heterogeneous networks},
  author={V. Colizza and R. Pastor-Satorras and A. Vespignani},
  journal={Nat. Phys.},
  volume={3},
  number={4},
  pages={276--282},
  year={2007}
}

@article{colizza2008epidemic,
  title={Epidemic modeling in metapopulation systems with heterogeneous coupling pattern: Theory and simulations},
  author={V. Colizza and A. Vespignani},
  journal={J. Theor. Biol.},
  volume={251},
  number={3},
  pages={450--467},
  year={2008}
}

@article{aronna2024optimal,
  title={Optimal vaccination strategies on networks and in metropolitan areas},
  author={Aronna, M. S. and Moschen, L. M.},
  journal={Infectious Disease Modeling},
  volume={9},
  number={4},
  pages={1198--1222},
  year={2024},
  publisher={Elsevier}
}

\end{document}